\documentclass{amsart}
\usepackage{amsmath}
\usepackage{amsthm}
\usepackage{amssymb}
\usepackage{epsfig}

\newcommand{\4}{\#}
\newcommand{\eq}{\ensuremath{\stackrel{4}{=}}}
\newcommand{\eqq}{\ensuremath{\stackrel{2}{=}}}
\newcommand{\bm}[1]{\mbox{\boldmath $#1$}}
\newcommand{\ul}[1]{\underline{#1}}

\theoremstyle{plain}
\newtheorem{thm}{Theorem}
\newtheorem{cor}[thm]{Corollary}
\newtheorem{lem}[thm]{Lemma}
\theoremstyle{definition}
\newtheorem*{defn}{Definition}
\theoremstyle{remark}
\newtheorem{rem}[thm]{Remark}

\title{Domino Tiling Congruence Modulo 4}
\author{Bridget Eileen Tenner}
\address{Department of Mathematical Sciences, DePaul University, Chicago, Illinois}
\email{bridget@math.depaul.edu}

\begin{document}

\begin{abstract}
The number of domino tilings of a region with reflective symmetry across a line is combinatorially shown to depend on the number of domino tilings of particular subregions, modulo $4$.  This expands upon previous congruency results for domino tilings, modulo $2$, and leads to a variety of corollaries, including that the number of domino tilings of a $k \times 2k$ rectangle is congruent to $1 \bmod 4$.
\end{abstract}

\maketitle

\section{Introduction}

Much work has been done to enumerate the domino tilings of various 
regions, and related the number of perfect matchings of subgraphs of $\mathbb{Z}^2$.  When an exact enumeration has been elusive, or when considering more general types of regions, various aspects of the number of domino tilings have been studied.  For example, in \cite{pachter}, Pachter studied 
$2$-divisibility in tilings, and Cohn examined the number of domino tilings of a square under the $2$-adic metric in \cite{cohn}.  It was recently shown by this author in \cite{oddity} that the parity of the number of domino tilings of a region and of a subregion will be equal under certain conditions.  This parity result leads to the question of whether there may be other moduli for which the number of domino tilings of a region and of a subregion will have the same residue.  This article shows that for regions with reflective symmetry, there is such a result for the modulus $4$.  Several of the proofs in this article are reminiscent of those in \cite{oddity}, but the fourfold nature of this topic requires additional sophistication and technical aspects.

In \cite{propp}, Propp suggested finding a combinatorial argument for the 
fact that the number of domino tilings of a $k \times 2k$ rectangle is 
congruent to $1 \bmod 4$.  There are analytic methods for enumerating 
the domino tilings of a rectangle, such as Kasteleyn's formula (see 
\cite{kasteleyn}).  For example, if $a$ and $b$ are both even, this 
formula states that the number of domino tilings of an $a \times b$ 
rectangle is

\begin{equation*}
\prod_{i=1}^{a/2} \prod_{j=1}^{b/2} \left(4\cos^2 \frac{i\pi}{a+1} + 4\cos^2 \frac{j\pi}{b+1}\right).
\end{equation*}

\noindent However, in general, the residue of this number, modulo $4$, is not easy to compute analytically, and so a combinatorial method is desirable.

This paper shows combinatorially that the residue, modulo $4$, of the 
number of domino tilings of a region with reflective symmetry can be 
computed by looking at the analogous residue for certain subregions.  
This will, among other things, answer Propp's question by showing 
inductively that the number of domino tilings of a $k \times 2k$ rectangle 
is $1 \bmod 4$.  This is not the first time that domino tilings for 
regions with reflective symmetry have been examined.  In \cite{ciucu}, 
Ciucu studies the number of perfect matchings of a graph with reflective 
symmetry (the dual of such an object is a domino tiling of a region with 
reflective symmetry), and derives a factorization theorem for this number 
in the case where the axis of symmetry separates the graph.  However, this 
does not answer the question for the $k \times 2k$ rectangle: one symmetry 
axis does not separate the graph, and the other will separate only if $k$ 
is odd, but the domino tilings of the subgraphs required in the 
factorization theorem would themselves have to be enumerated as a separate 
problem.  In this paper, a region with a local property and a reflective 
axis of symmetry, 
whether or not it is separating, is shown to have the 
same number of domino tilings, modulo $4$, as particular subregions.  
This is similar to Ciucu's work, in that it considers symmetric regions, 
but it is applicable in different circumstances, and describes a different 
aspect of this enumeration.  It should also be emphasized that besides the 
symmetry and local property, no further assumptions are made about the 
region.

The main terminology and notation used throughout the paper will be 
introduced in Section~\ref{sec:defn}.  This section will also present the 
parity theorem of \cite{oddity}, which is the precursor to the main result 
in this paper, and which plays a prominent role in the proof of that 
result.  Section~\ref{sec:4ity} consists entirely of the main result of 
this article, together with its proof.  The arguments in this proof are inductive and resemble those of the 
aforementioned parity theorem.  Finally, Section~\ref{sec:apps} discusses 
a variety of applications of Theorem~\ref{thm:4ity theorem}, one of which 
resolves Propp's question about the $k \times 2k$ rectangle.

\section{Definitions and notation}\label{sec:defn}

The main terminology of the paper is described here.  Some of these definitions and notation can be cumbersome, however, as in \cite{oddity}, they permit the results to be as general as possible.  Much of this terminology is also used in \cite{oddity}, although the class of regions considered there is broader, and hence the result is not as specialized.  The primary difference between the regions considered in \cite{oddity} and the regions considered here is the additional requirement of reflective symmetry.

\begin{defn}
A \emph{region} is the dual of a finite connected induced subgraph of $\mathbb{Z}^2$.
\end{defn}

\begin{defn}
The number of domino tilings of a region $R$ is denoted $\#R$.
\end{defn}

The notation $a \stackrel{m}{=} b$ will indicate that $a \equiv b \bmod m$.

A basic lemma, which is key in the proof of the main theorem, is the following.  The proof is straightforward, and thus is omitted here.

\begin{lem}\label{lem:mod4mod2}
For $r \in \mathbb{Z}$, set $r_m := r \ (\bmod \ m)$ to be an integer in 
the set $[0,m-1]$.  Then for all $r$,
\begin{equation*}
2r_4 \eq 2r_2.
\end{equation*}
\end{lem}

All tilings discussed in this article are domino tilings, thus the word ``domino'' may be omitted.  As in \cite{oddity}, when the configuration of only part of a region is under consideration, a drawing of the region may only include this portion, while the undrawn part of the region is arbitrary.  Shading is used to indicate when part of a region has been removed.

\begin{defn}
A region $R$ has an \emph{$(\{s,t\}; 1)$-corner} if there is a convex corner in $R$ where the segments bounding this corner have lengths $s$ and $t$.  For $p > 1$ and $\min\{s,t\} \ge 2$, an \emph{$(\{s,t\}; p)$-corner} is a $(\{1,s\};1)$-corner, a $(\{1,t\};1)$-corner, and $p-2$ distinct $(\{1,1\};1)$-corners configured as in Figure~\ref{figure:stp}.
\end{defn}

\begin{figure}[htbp]
\centering
\input{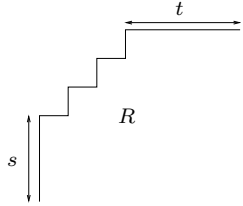}
\caption{An $(\{s,t\};4)$-corner.}\label{figure:stp}
\end{figure}

\begin{defn}
If the segment of length $s$ in an $(\{s,t\};p)$-corner forms an $(\{s,t'\}; p')$-corner at its other endpoint, then each of these corners is \emph{walled} at $s$.
\end{defn}

\begin{defn}
An \emph{$(\{i,j\};p)$-strip} is a subregion of $i+j+2p-3$ squares that has an $(\{i,j\};p)$-corner.
\end{defn}

The local property required for the main theorem, Theorem~\ref{thm:4ity theorem}, states that a particular subregion avoids holes in certain places.  A notion of completion describes this property, made precise below.  To simplify things, one could consider only regions without holes, although this would ignore a large class of regions for which the results are also true.

\begin{defn}
An $(\{s,t\};p)$-corner in a region $R$ is \emph{$2$-complete} if $2 \le \min\{s,t\}$.  For $2 < i \le \min\{s,t\}$, the corner is \emph{$i$-complete} if the following conditions are met:

\begin{enumerate}
\item Let $C$ be the $(\{i,i\}; p)$-strip in the $(\{s,t\};p)$-corner.  Let $x$ and $y$ be the two squares adjacent to the ends of $C$ but not along the edges forming the $(\{s,t\};p)$-corner.  If either $x$ or $y$ is in $R$, then the $(\{i-1,i-1\};p)$-strip between $x$ and $y$, inclusively, all of whose squares are adjacent to $C$, must also be a subregion of $R$.
\begin{equation*}
\input{completestp.pstex_t}
\end{equation*}
\item Consider removing $C$ from $R$.  This leaves some $(\{s',t'\};p)$-corner in the resulting subregion.  If $2 \le i-2 \le \min\{s',t'\}$, then this corner must be $j$-complete for $j = 2, \ldots, i-2$.
\end{enumerate}

\end{defn}

To decide if an $(\{s,t\};p)$-corner is $k$-complete, the largest potential subregion of $R$ that must be examined is depicted in Figure~\ref{fig:completeness subregion}, where there are $\lceil k/2 \rceil$ strips: one $(\{k,k\};p)$-strip, one $(\{k-1,k-1\};p)$-strip, one $(\{k-3,k-3\};p)$-strip, one $(\{k-5,k-5\};p)$-strip, $\ldots$, concluding with a $(\{3,3\};p)$-strip if $k$ is even, or a $(\{2,2\};p)$-strip if $k$ is odd.
\begin{figure}[htbp]
\epsfig{file=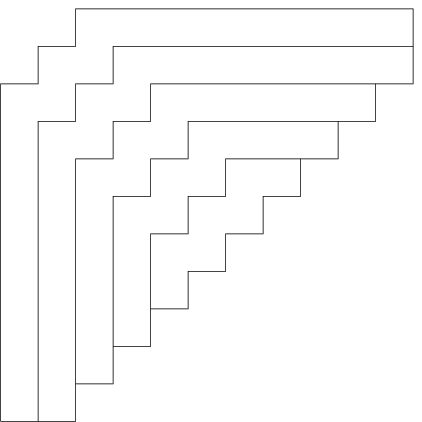}
\caption{The largest subregion examined when determining if a 
$(\{s,t\};3)$-corner is $9$-complete.} \label{fig:completeness subregion}
\end{figure}

\begin{defn}
For $2 \le k \le \min\{s,t\}$, an $(\{s,t\};p)$-corner is \emph{complete up to $k$} if that corner is $i$-complete for $i = 2, \ldots, k$.
\end{defn}

Examples of corners that are complete up to $3$ and $4$ can be found in \cite{oddity}, and are not repeated here.

The main result of \cite{oddity}, and a useful tool in the proof of the main theorem of this paper, is restated below in Theorem~\ref{thm:parity theorem}.

Throughout this article, if a region is depicted having a shaded subregion (as in equation~\eqref{eqn:2-open}), then that shading is understood to indicate a subtiling.  The figures in Theorems~\ref{thm:parity theorem} and~\ref{thm:4ity theorem} depict $(\{s,t\};p)$-corners with $p=4$ to indicate the general result.  When $p=1$, the figures in equation~\eqref{eqn:2-open}, for example, would each consist of a single right-angle and a shaded $(\{k,k+1\};1)$-strip (once with each orientation).

\begin{thm}[\cite{oddity}]\label{thm:parity theorem}
Suppose that a region $R$ has an $(\{s,t\};p)$-corner.  For any $2 \le k \le \min\{s,t\}$, if this corner is complete up to $k$, then
\begin{equation}\label{eqn:2-open}
\# R \stackrel{2}{=} \# \hspace{.05in}
\parbox{1.5in}{\begin{picture}(0,0)%
\epsfig{file=kk+1.pstex}%
\end{picture}%
\setlength{\unitlength}{1184sp}%
\begingroup\makeatletter\ifx\SetFigFont\undefined%
\gdef\SetFigFont#1#2#3#4#5{%
  \reset@font\fontsize{#1}{#2pt}%
  \fontfamily{#3}\fontseries{#4}\fontshape{#5}%
  \selectfont}%
\fi\endgroup%
\begin{picture}(5487,5400)(2926,-5773)
\put(2926,-3811){\makebox(0,0)[lb]{\smash{{\SetFigFont{8}{9.6}{\familydefault}{\mddefault}{\updefault}$k$}}}}
\put(6151,-661){\makebox(0,0)[lb]{\smash{{\SetFigFont{8}{9.6}{\familydefault}{\mddefault}{\updefault}$k+1$}}}}
\end{picture}%
}
 + \#  \hspace{.05in}
 \parbox{1.6in}{\begin{picture}(0,0)%
\epsfig{file=k+1k.pstex}%
\end{picture}%
\setlength{\unitlength}{1184sp}%
\begingroup\makeatletter\ifx\SetFigFont\undefined%
\gdef\SetFigFont#1#2#3#4#5{%
  \reset@font\fontsize{#1}{#2pt}%
  \fontfamily{#3}\fontseries{#4}\fontshape{#5}%
  \selectfont}%
\fi\endgroup%
\begin{picture}(6087,5400)(2326,-5773)
\put(6151,-661){\makebox(0,0)[lb]{\smash{{\SetFigFont{8}{9.6}{\familydefault}{\mddefault}{\updefault}$k$}}}}
\put(2326,-4111){\makebox(0,0)[lb]{\smash{{\SetFigFont{8}{9.6}{\familydefault}{\mddefault}{\updefault}$k+1$}}}}
\end{picture}%
}.
\end{equation}
\noindent If $p = 1$, then \eqref{eqn:2-open} also holds for $k=1$.  Furthermore, for any $p$, if $s \le t$, the corner is complete up to $s$, and the corner is walled at $s$, then 
\begin{equation}\label{eqn:2-wall} 
\# R \stackrel{2}{=} \#  \hspace{.05in}
\parbox{1.5in}{\begin{picture}(0,0)%
\epsfig{file=ss+1.pstex}%
\end{picture}%
\setlength{\unitlength}{1184sp}%
\begingroup\makeatletter\ifx\SetFigFont\undefined%
\gdef\SetFigFont#1#2#3#4#5{%
  \reset@font\fontsize{#1}{#2pt}%
  \fontfamily{#3}\fontseries{#4}\fontshape{#5}%
  \selectfont}%
\fi\endgroup%
\begin{picture}(5487,4200)(2926,-4573)
\put(2926,-3811){\makebox(0,0)[lb]{\smash{{\SetFigFont{8}{9.6}{\familydefault}{\mddefault}{\updefault}$s$}}}}
\put(6151,-661){\makebox(0,0)[lb]{\smash{{\SetFigFont{8}{9.6}{\familydefault}{\mddefault}{\updefault}$s+1$}}}}
\end{picture}%
}.
\end{equation}
\end{thm}

This paper is concerned with regions that have reflective symmetry.  This added assumption is what allows for the refined understanding of $\# R$.

\begin{defn}
A region $R$ is \emph{reflective} if it has a reflective axis of symmetry.
\end{defn}

\begin{defn}
A pair of $(\{s,t\};p)$-corners in a reflective region $R$ is 
a \emph{reflective pair} if the two corners are images of each other under 
reflection 
across the region's axis of symmetry.  Moreover, when determining if a reflective pair of corners is complete up to $k$, it is further assumed that the subregions examined to determine completeness, as depicted in Figure~\ref{fig:completeness subregion}, are not overlapping.  This means that the symmetry axis does not intersect any squares in the subregion depicted in Figure~\ref{fig:completeness subregion}.
\end{defn}

In all figures, reflective regions will be drawn so that the symmetry axis is vertical, and it will be depicted by a dashed line.  Although the figures may suggest that the symmetry axis is parallel to one set of grid lines in the region, the exact orientation of the left and right sides of a reflective region are not restricted, unless they are shown to connect across the symmetry axis, as in equation~\eqref{eqn:k2kk}.

\begin{rem}\label{rem:reflections}
It is helpful to observe that if two regions are reflections of each other across a line, they have the same number of tilings.
\end{rem}

\section{Reflective regions, modulo 4}\label{sec:4ity}

The main result of this article resembles the parity theorem of \cite{oddity} (Theorem~\ref{thm:parity theorem} here), although the region $R$ is additionally required to be reflective and there are more subregions on the right side of the equation.  Some of the ideas fundamental to the proof of Theorem~\ref{thm:4ity theorem} are similar to those in the proof of Theorem~\ref{thm:parity theorem}, which explains the necessary overlap of terminology in the previous section.

\begin{thm}\label{thm:4ity theorem}
Suppose that a region $R$ is reflective, and that $R$ has a reflective pair of $(\{s,t\};p)$-corners.  For any $2 \le k \le \min\{s,t\}$, if these corners are complete up to $k$, then
\begin{equation}\label{eqn:4-open}\begin{split}
\4 R \eq & \hspace{.05in} \4 \hspace{.05in}
\parbox{1.9in}{\resizebox{1.9in}{!}{\input{shorthh.pstex_t}}}
+ \4 \hspace{.05in} 
\parbox{1.75in}{\resizebox{1.75in}{!}{\input{shortvv.pstex_t}}}\\
&+ 2 \4 \hspace{.05in}
\parbox{1.9in}{\resizebox{1.9in}{!}{\input{shorthv.pstex_t}}}.
\end{split}\end{equation}
\noindent If $p = 1$, then \eqref{eqn:4-open} also holds for $k=1$.  Furthermore, for any $p$, if $s \le t$, the corners are complete up to $s$, and the corners are walled at $s$, then 
\begin{equation}\label{eqn:4-wall} 
\4 R \eq  \4 \hspace{.05in} \parbox{1.75in}{\resizebox{1.75in}{!}{\input{shortvv-wall.pstex_t}}}.
\end{equation}
\noindent (Note that the exact orientation of the reflective pair of corners, with respect to the axis of symmetry, is not specified.)
\end{thm}

\begin{proof}

The proof of this theorem is primarily inductive.  First the theorem will be proved for $p=1$, by induction on $k$.  Subsequently, the proof will be completed by inducting on $p$.  Note that for fixed $k$ and $p$, equation \eqref{eqn:4-open} trivially implies equation \eqref{eqn:4-wall}, since two of the regions pictured in \eqref{eqn:4-open} are impossible if the corners are walled at $s$.

Suppose that $p=1$.  If $k=1$, then \eqref{eqn:4-open} trivially holds, by Remark~\ref{rem:reflections} and the fact that there are two ways to place a domino in each of the two $(\{s,t\};1)$-corners under consideration.  Observe that \eqref{eqn:4-open} also holds for $k=2$, because there are eight ways to tile the two $(\{2,2\};1)$-corners, and four of these eight tile the same subregion of $R$: a $2 \times 2$ block in each of the two corners.

Now assume that the theorem holds for all $1 \le k < K \le \min\{s,t\}$, and suppose that $R$ has a reflective pair of $(\{s,t\};1)$-corners complete up to $K$.  These corners must also be complete up to $K-1$, so \eqref{eqn:4-open} holds for $k = K-1$, as indicated in equation~\eqref{eqn:Kinduction}.

\begin{equation}\label{eqn:Kinduction}\begin{split}
\4 R \eq & \hspace{.05in} \4 \hspace{.05in}
\parbox{1.75in}{\resizebox{1.75in}{!}{\input{K1shorthh.pstex_t}}}
+ \4 \hspace{.05in}
\parbox{1.95in}{\resizebox{1.95in}{!}{\input{K1shortvv.pstex_t}}}\\
&+ 2 \4 \hspace{.05in}
\parbox{1.85in}{\resizebox{1.85in}{!}{\input{K1shorthv.pstex_t}}}
\end{split}\end{equation}

To extend a tiled region in \eqref{eqn:Kinduction} to cover the two $(\{K,K\},1)$-corners, one tile must be added to each corner, and each added tile can have two orientations: horizontal or vertical.  Each of the four configurations in \eqref{eqn:Kinduction} (counting the doubled third term twice) leads to four tilings, giving a total of sixteen possible tilings (counting repetitions).

At this point, it is helpful to introduce notation.  In each figure in equation~\eqref{eqn:Kinduction}, one corner is depicted to the left of the symmetry axis and one corner is depicted to the right of the axis.  The configuration of the left corner will be described in bold face, and the configuration of the right corner will be described with an underline.  In the tilings of \eqref{eqn:Kinduction}, if the longer leg of the tiled corner is horizontal (respectively, vertical), this will be called an $H$ (respectively, $V$) tiling.  If the tile added in order to cover the $(\{K,K\},1)$-corner is horizontal (respectively, vertical), this will be called an $h$ (respectively, $v$) tile.  For example, the third figure in equation~\eqref{eqn:Kinduction} is denoted $\bm{V}\ul{H}$.

Therefore, the sixteen tilings referred to above are described in the following table:
\begin{equation}\label{eqn:table of tilings}
\begin{array}{r|ccccc}
\text{Tiling from \eqref{eqn:Kinduction}} & \text{Add:} & \bm{h} \text{ and } \ul{h} & \bm{h} \text{ and } \ul{v} & \bm{v} \text{ and } \ul{h} & \bm{v} \text{ and } \ul{v}\\
\hline
\bm{V}\ul{V} & & \bm{Vh}\ul{Vh} & \bm{Vh}\ul{Vv} & \bm{Vv}\ul{Vh} & \bm{Vv}\ul{Vv}\\
\bm{H}\ul{H} & & \bm{Hh}\ul{Hh} & \bm{Hh}\ul{Hv} & \bm{Hv}\ul{Hh} & \bm{Hv}\ul{Hv}\\
\text{(twice) } \bm{V}\ul{H} & & \bm{Vh}\ul{Hh} & \bm{Vh}\ul{Hv} & \bm{Vv}\ul{Hh} & \bm{Vv}\ul{Hv}\\
\end{array}
\end{equation}

The reflective symmetry of $R$ and Remark~\ref{rem:reflections} indicate that \eqref{eqn:table of tilings} still includes duplications: $\# \bm{Vh}\ul{Vv} = \#\bm{Vv}\ul{Vh}$ and $\# \bm{Hh}\ul{Hv} = \# \bm{Hv}\ul{Hh}$.

Combining these facts yields the following identity:
\begin{equation}\label{eqn:variable equation}\begin{split}
\4 R \eq & \  \4 \bm{Vh}\ul{Vh} + \4 \bm{Hv}\ul{Hv} + 2\4 \bm{Vh}\ul{Hv}\\
&+ 2 \4 \bm{Vh}\ul{Vv} + \4 \bm{Vv}\ul{Vv}\\
&+ \4 \bm{Hh}\ul{Hh} + 2 \4 \bm{Hh}\ul{Hv}\\
&+ 2 \4 \bm{Vh}\ul{Hh} + 2 \4 \bm{Vv}\ul{Hh} + 2 \4 \bm{Vv}\ul{Hv}.
\end{split}\end{equation}

\noindent Notice that if the only terms on the right side of equation~\eqref{eqn:variable equation} were the first three terms, the induction on $K$ would be complete.

Observe that equation~\eqref{eqn:2-wall} of Theorem~\ref{thm:parity 
theorem} applies to the left sides of $\bm{Hh}\ul{Hv}$ and 
$\bm{Vv}\ul{Hv}$, as in Figure~\ref{fig:HhHv-par}.  Additionally, 
Lemma~\ref{lem:mod4mod2} permits the transition from working modulo $4$ to 
working modulo $2$, since the terms $\4 \bm{Hh}\ul{Hv}$ and $\4 
\bm{Vv}\ul{Hv}$ each have coefficient $2$ in \eqref{eqn:variable 
equation}.

\begin{figure}[htbp]
\centering
\input{HhHv-par.pstex_t}
\caption{Application of Theorem~\ref{thm:parity theorem} to the left side of $\bm{Hh}\ul{Hv}$.}\label{fig:HhHv-par}
\end{figure}

After applying Theorem~\ref{thm:parity theorem} to each of 
$\bm{Hh}\ul{Hv}$ and $\bm{Vv}\ul{Hv}$, the resulting region is the same.  
Therefore, $2 \4 \bm{Hh}\ul{Hv} + 2 \4 \bm{Vv}\ul{Hv}$ is a multiple of 4, and thus these terms can be ignored in \eqref{eqn:variable equation}.  Likewise $2 \4 \bm{Vh}\ul{Vv} + 2 \4 
\bm{Vh}\ul{Hh}$ contributes a multiple of $4$ to the right side of 
\eqref{eqn:variable equation}, and can be ignored as well.

All that remains in \eqref{eqn:variable equation}, besides the first three terms, is
\begin{equation*}
\4 \bm{Vv}\ul{Vv} + \4 \bm{Hh}\ul{Hh} + 2 \4 \bm{Vv}\ul{Hh}.
\end{equation*}
\noindent The regions $\bm{Vv}\ul{Vv}$ and $\bm{Hh}\ul{Hh}$ are both reflective with the same symmetry axis as $R$.  Therefore, by induction, they obey equation~\eqref{eqn:4-wall}.  Similarly, equation~\eqref{eqn:2-wall} can be applied twice to $\bm{Vv}\ul{Hh}$, again using Lemma~\ref{lem:mod4mod2}.  After applying these results as described, the three regions are identical.  Therefore, since $\4 \bm{Vv}\ul{Hh}$ is multiplied by $2$ in \eqref{eqn:variable equation}, these three terms contribute a multiple of $4$ to the sum, and consequently can be ignored.

Therefore, equation~\eqref{eqn:variable equation} reduces to
\begin{equation*}
\4 R \eq \ \4 \bm{Vh}\ul{Vh} + \4 \bm{Hv}\ul{Hv} + 2 \4 \bm{Vh}\ul{Hv},
\end{equation*}
\noindent which, by Remark~\ref{rem:reflections}, is the statement of the theorem for $p=1$ and $k = K$.  This completes the first part of the induction.

To complete the proof, assume that the theorem is true for all $1 \le p < P$, and suppose that $R$ has a reflective pair of $(\{s,t\};P)$-corners.  Consider the pair of $(\{s,1\};1)$-corners in these $(\{s,t\};P)$-corners.  A tile can be placed in each of these corners either horizontally or vertically.  A horizontal tile forces $P-1$ additionally horizontal tiles above it, and leaves a $(\{s-1,2\};1)$-corner below it, while a vertical tile creates a $(\{3,t\};P-1)$-corner to the right of the tile.
If a horizontal tile is placed in one of the two $(\{s,1\};1)$-corners, and a vertical tile in the other, the resulting figure is the reflection across the symmetry axis of the configuration where the two tiles are switched.  Therefore
\begin{equation}\label{eqn:Phv figure}
\#R = \# \parbox{1.25in}{\epsfig{file=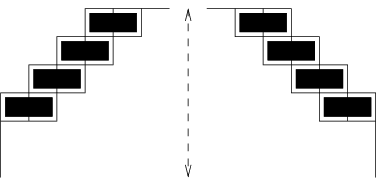,width=1.25in}} + \#\parbox{1.25in}{\epsfig{file=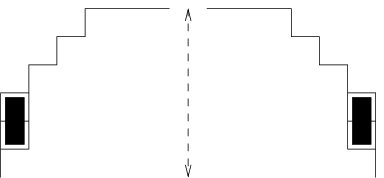,width=1.25in}} + 2\# \parbox{1.25in}{\epsfig{file=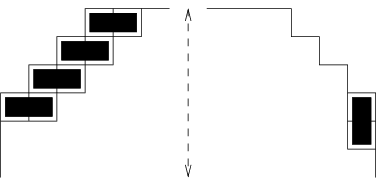,width=1.25in}}.
\end{equation}

Note that the first and second figures in equation~\eqref{eqn:Phv figure} are reflective.  Additionally, by the induction hypothesis, the theorem holds for $(\{s-1,2\};1)$-corners and for $(\{3,t\};P-1)$-corners, so the theorem holds for these corners in the first and second figures of \eqref{eqn:Phv figure}.  Likewise, Theorem~\ref{thm:parity theorem} applies to the left and right sides of the third figure in \eqref{eqn:Phv figure}.  Applying these theorems, as well as Lemma~\ref{lem:mod4mod2}, and taking into account when two configurations are identical except for reflection across the axis of symmetry yields the following:
\begin{equation}\label{eqn:Phv figure - detail}\begin{split}
\4 R \eq & \hspace*{.05in} \4 \parbox{1in}{\epsfig{file=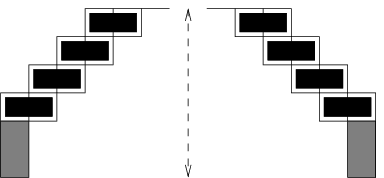,width=1in}} + \4 \parbox{1in}{\epsfig{file=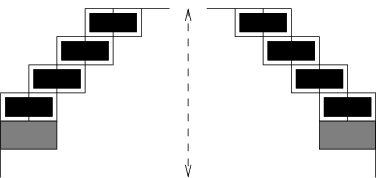,width=1in}} + 2 \4 \parbox{1in}{\epsfig{file=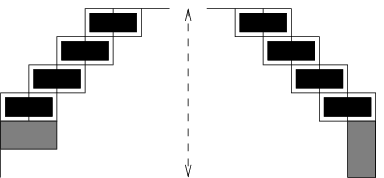,width=1in}}\\
& + \4 \parbox{1in}{\epsfig{file=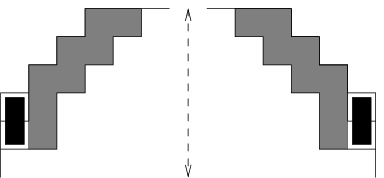,width=1in}} + \4 \parbox{1in}{\epsfig{file=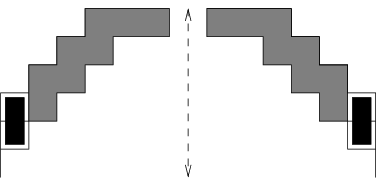,width=1in}} + 2 \4 \parbox{1in}{\epsfig{file=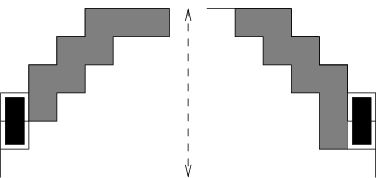,width=1in}}\\
& + 2 \4 \parbox{1in}{\epsfig{file=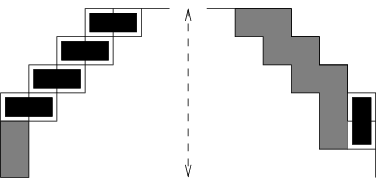,width=1in}} + 2 \4 \parbox{1in}{\epsfig{file=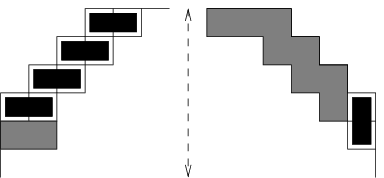,width=1in}} + 2 \4 \parbox{1in}{\epsfig{file=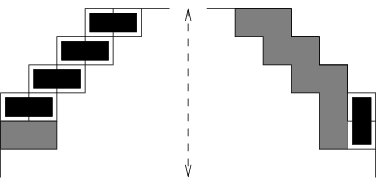,width=1in}}\\
& + 2 \4 \parbox{1in}{\epsfig{file=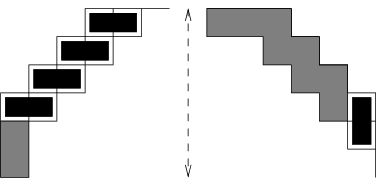,width=1in}}.
\end{split}\end{equation}

Since everything in equation~\eqref{eqn:Phv figure - detail} is taken modulo $4$, those terms that appear four times can be ignored.  There are three such quadruples, and the remaining equation is exactly equation~\eqref{eqn:4-open} for $k=2$ and $p=P$.

To complete the proof, assume the result holds for this $P$ and for all $2 \le k < K \le \min\{s,t\}$, where the $(\{s,t\};P)$-corners under consideration are complete up to $K$.  These corners must also be complete up to $K-1$, so the result holds for $k=K-1$.  An argument analogous to the previous discussion for $p=1$ will complete the proof.
\end{proof}

\section{Applications of the congruency result}\label{sec:apps}

Under certain conditions, Theorem~\ref{thm:4ity theorem} relates the number of domino tilings of a region modulo $4$, to that aspect of the number of domino tilings of three particular subregions.  Thus, $\4 R$ may be determined by iteratively determining $\4 R_i$ for these subregions $R_i \subset R$.  This is not necessarily an attractive prospect, and so it is helpful to highlight some situations when Theorem~\ref{thm:4ity theorem} can be used to great effect.

One easy application of Theorem~\ref{thm:4ity theorem} is the following.

\begin{cor}\label{cor:double wall}
Suppose that a region $R$ is reflective, and that $R$ has a reflective pair of $(\{s,s\};p)$-corners.  If these corners are complete up to $s$ and each walled at $s$ along both sides, then $\# R$ is a multiple of $4$.
\end{cor}

\begin{proof}
The three configurations on the right side of equation~\eqref{eqn:4-open} are impossible, so all of these terms contribute zero to the sum.
\end{proof}

More can be said about $\4 R$ if the reflective pair of corners share a longer side crossing the axis of symmetry.  This occurs, for example, in the $k \times 2s$ and $k \times (2s+1)$ rectangles, where $k \ge s$.

\begin{cor}\label{cor:k2kk}
Suppose that a region $R$ is reflective, and that $R$ has a reflective pair of $(\{s,s\};p)$-corners.  If these corners are complete up to $s$ and $R$ has an edge of length $2s$ which is comprised of one segment of length $s$ from each corner (hence the axis of symmetry bisects this segment of length $2s$), then
\begin{equation}\label{eqn:k2kk}
\4 R \eq \4 \parbox{1.75in}{\resizebox{1.75in}{!}{\input{k2kkshort.pstex_t}}}
+ \4 \parbox{2.04in}{\resizebox{2.04in}{!}{\input{k2kklong.pstex_t}}}.
\end{equation}
\end{cor}

\begin{proof}
First apply Theorem~\ref{thm:4ity theorem} for $k=s-1$.  Then, to each of the configurations described in equation~\eqref{eqn:4-open}, add two tiles (or, in one situation, a single tile) so that the two $(\{s,s\};p)$-corners are each completely tiled.  Combining those configurations that are obtained by reflection across the axis of symmetry yields
\begin{equation}\label{eqn:k2kk possibilities}\begin{split}
\4 R \eq & \hspace{.05in} \4 \parbox{1.1in}{\resizebox{1.1in}{!}{\input{k2kk1.pstex_t}}}
+  \4 \parbox{1.1in}{\resizebox{1.1in}{!}{\input{k2kk2.pstex_t}}}
+ \4 \parbox{1.1in}{\resizebox{1.1in}{!}{\input{k2kk5.pstex_t}}}\\
& + 2 \4 \parbox{1.1in}{\resizebox{1.1in}{!}{\input{k2kk6.pstex_t}}}
+ \4 \parbox{1.1in}{\resizebox{1.1in}{!}{\input{k2kk7.pstex_t}}}
+  2 \4 \parbox{1.1in}{\resizebox{1.1in}{!}{\input{k2kk3.pstex_t}}}\\
& + 2 \4 \parbox{1.1in}{\resizebox{1.1in}{!}{\input{k2kk4.pstex_t}}}.
\end{split}\end{equation}

The second and third figures on the right side of equation~\eqref{eqn:k2kk possibilities} are reflective, so Theorem~\ref{thm:4ity theorem}, in particular, equation~\eqref{eqn:4-wall}, applies to each.  In fact, both applications give the same configuration.  Applying Theorem~\ref{thm:parity theorem} (and Lemma~\ref{lem:mod4mod2}) to each side of the doubled seventh term in \eqref{eqn:k2kk possibilities} gives the same configuration again.  Thus this resulting term appears four times, so it can be ignored.

Similarly, apply Theorem~\ref{thm:parity theorem} and Lemma~\ref{lem:mod4mod2} to each side of the doubled sixth figure on the right side of \eqref{eqn:k2kk possibilities}, and do likewise to the doubled fourth term, first to the right corner and then to the left.  These yield two pairs of figures which are reflections of each other across the axis of symmetry.  Thus they have the same number of domino tilings, and so this again gives a multiple of four, which can be ignored.

Therefore the only figures on the right side of equation~\eqref{eqn:k2kk possibilities} which are left are the first and fifth, which is the statement of the corollary.
\end{proof}

\begin{cor}\label{cor:k2k+1k}
Suppose that a region $R$ is reflective, and that $R$ has a reflective pair of $(\{s,s\};p)$-corners.  If these corners are complete up to $s$ and $R$ has an edge of length $2s+1$ which is comprised of one segment of length $s$ from each corner and a segment of length $1$ between them (hence the axis of symmetry bisects this segment of length $2s+1$), then
\begin{equation}\label{eqn:k2k+1k}
\4 R \eq  \4 \parbox{1.95in}{\resizebox{1.95in}{!}{\input{k2k+1klong.pstex_t}}}
+ 2 \4 \parbox{1.85in}{\resizebox{1.85in}{!}{\input{k2k+1kshort.pstex_t}}}.
\end{equation}
\end{cor}

\begin{proof}
This follows immediately from equation~\eqref{eqn:4-open} of Theorem~\ref{thm:4ity theorem}, since the second tiling in \eqref{eqn:4-open} is impossible, and the first tiling in \eqref{eqn:4-open} forces a vertical tile between the two pieces, and equation~\eqref{eqn:4-wall} applies to the resulting region.
\end{proof}

As in \cite{oddity}, consider the following types of regions.

\begin{defn}
Let $T(i,j,p)$ be the region with $i+p-1$ centered rows of lengths $j, j+2, \ldots, j+2(p-1), \ldots, j+2(p-1)$ from top to bottom.
\end{defn}

\begin{defn}
Let $D(i,j,p)$ be the region with $i+2(p-1)$ centered rows of lengths $j, j+2, \ldots, j+2(p-1), \ldots, j+2(p-1), \ldots, j+2, j$ from top to bottom.  That is, if $i$ is even, $D(i,j,p)$ consists of two copies of $T(i/2,j,p)$ that have been glued together along the edge of length $j+2(p-1)$.
\end{defn}

For example, the Aztec diamond of order $p$ is the region $D(2,2,p)$, and the regions $T(2,5,4)$ and $D(2,5,4)$ are depicted in Figure~\ref{fig:TD ex}.

\begin{figure}[htbp]
\centering
$\begin{array}{c@{\hspace{.5in}}c}
\multicolumn{1}{l}{\mbox{\bf{(a)}}} & \multicolumn{1}{l}{\mbox{\bf{(b)}}}\\
[-.25cm]
\epsfig{file=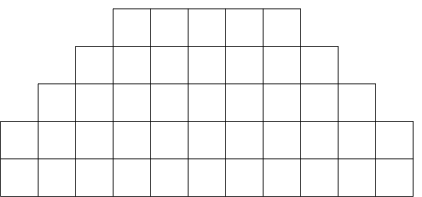,width=1.25in} &\epsfig{file=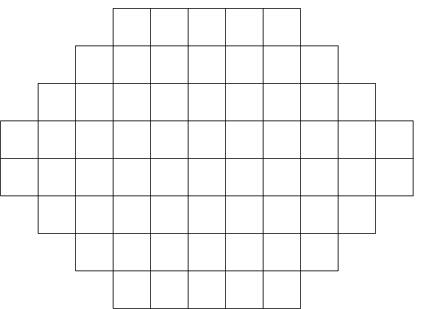,width=1.25in}
\end{array}$
\caption{$\bf{(a)}$ The region $T(2,5,4)$.  $\bf{(b)}$ The region $D(2,5,4)$, which is two copies of $T(1,5,4)$ glued together along the edges of length $11$.}\label{fig:TD ex}
\end{figure}

The following results concerning the parities of $\# T(i,j,p)$ and of $\# D(i,j,p)$ appeared in \cite{oddity}, although Corollary~\ref{cor:oddity T}(a) and Corollary~\ref{cor:oddity D}(c) were incorrectly stated for $p=1$.

\begin{cor}[\cite{oddity}]\label{cor:oddity T}
\ 
\begin{enumerate}
\item $\# T(k,2k-1,p) \eqq \left\{ \begin{array} {c@{\quad:\quad}l}
1 & p=1 \text{ and } k \text{ is even};\\ 0 & \text{otherwise}. \end{array} \right.$.
\item $\# T(k,2k,p) \eqq 1$.
\item $\# T(k,2k+1,p) \eqq 0$.
\item $\# T(k,2k+2,p) \eqq 1$.
\end{enumerate}
\end{cor}

\begin{cor}[\cite{oddity}]\label{cor:oddity D}
\begin{enumerate}
\item $\# D(k,k,p) \eqq 0$.
\item $\# D(k,k+1,p) \eqq 1$.
\item $\# D(k,k+2,p) \eqq \left\{ \begin{array} {c@{\quad:\quad}l}
1 & p=1 \text{ and } $k$ \text{ is even};\\ 0 & \text{otherwise}. \end{array} \right.$
\end{enumerate}
\end{cor}

These results can be refined using Theorem~\ref{thm:4ity theorem}.  For example, Corollary~\ref{cor:oddity T}(b) states that $\# T(k,2k,p) \eqq 1$.  Thus $\4 T(k,2k,p) \eq 1 \text{ or } 3$.  In fact, the exact residue modulo $4$ is easy to determine from previous results in this paper.

\begin{cor}\label{cor:T}
\ 
\begin{enumerate}
\item $\4 T(k,2k-1,p) \eq \left\{ \begin{array} {c@{\quad:\quad}l}
1 & p=1 \text{ and } k \eq 0;\\ 3 & p=1 \text{ and } k \eq 2;\\ 0 & \text{otherwise}. \end{array} \right.$
\item $\4 T(k,2k,p) \eq 1$ for all $k, p \ge 1$.
\item $\4 T(k,2k+1,p) \eq 0$ for all $k, p \ge 1$.
\item $\4 T(k,2k+2,p) \eq 1$ for all $k, p \ge 1$.
\end{enumerate}
\end{cor}

\begin{proof}
\ 
\begin{enumerate}
\item By Corollary~\ref{cor:k2k+1k} and Theorem~\ref{thm:parity theorem},
\begin{equation*}
\4 T(k,2k-1,p) \eq \4 3T(k-2, 2(k-2)-1,p).
\end{equation*}
The base cases are as follows.  The region $T(1,1,p)$ has no domino tilings for any $p$.  Similarly, iteratively applying Theorem~\ref{thm:4ity theorem} to $T(2,3,p)$ for $p>1$ shows that $\4 T(2,3,p) \eq 0$ for $p>1$.  On the other hand, $\# T(2,3,1) = 3$.
\item By Corollary~\ref{cor:k2kk}, $\4 T(k,2k,p) \eq \4 T(k-1, 2(k-1), p)$.  It is easy to see that $\# T(1,2,p) = 1$ for all $p$.
\item This follows from Corollary~\ref{cor:k2k+1k}, since both configurations in equation~\eqref{eqn:k2k+1k} are impossible.
\item Theorem~\ref{thm:4ity theorem}, in particular equation~\eqref{eqn:4-wall}, states that $\4 T(k,2k+2,p) \eq \4 T(k-1, 2(k-1) + 2, p)$.  It is straightforward to see that $\# T(1,4,p) = 1$ for all $p$.
\end{enumerate}
\end{proof}

By setting $p=1$, Corollary~\ref{cor:T}(b) addresses Propp's request in \cite{propp} for a combinatorial explanation that the number of domino tilings of a $k \times 2k$ rectangle is always $1 \bmod 4$.

\begin{cor}\label{cor:n2n}
The number of domino tilings of a $k \times 2k$ rectangle is congruent to $1 \bmod 4$ for all $k$.
\end{cor}

Corollary~\ref{cor:oddity D} can also be refined to determine the residues modulo $4$.
 
\begin{cor}\label{cor:D}
\ 
\begin{enumerate}
\item $\4 D(k,k,p) \eq \left\{ \begin{array} {c@{\quad:\quad}l}
2 & k=2 \text{ and } p=1;\\ 0 & \text{otherwise}. \end{array} \right.$
\item $\4 D(k,k+1,p) \eq \left\{ \begin{array} {c@{\quad:\quad}l}
1 & p \text{ is odd and } \lfloor (k+2)/4 \rfloor \text{ is even};\\
3 & p \text{ is odd and } \lfloor (k+2)/4 \rfloor \text{ is odd};\\
1 & p \text{ is even and } \lfloor k/4 \rfloor \text{ is even};\\
3 & p \text{ is even and } \lfloor k/4 \rfloor \text{ is odd}. \end{array} \right.$
\item $\4 D(k,k+2,p) \eq \left\{ \begin{array} {c@{\quad:\quad}l}
2 & p=2 \text{ and } k \text{ is even};\\ 1 & p=1, \ k \text{ is even, and } \lceil k/4 \rceil \text{ is odd};\\ 3 & p=1, \ k \text{ is even, and } \lceil k/4 \rceil \text{ is even};\\ 0 & \text{otherwise}. \end{array} \right.$
\end{enumerate}
\end{cor}

\begin{proof}
\ 
\begin{enumerate}
\item If $k$ is odd, the region $D(k,k,p)$ contains an odd number of squares, so $\# D(k,k,p) = 0$.  Thus it remains only to consider $D(2k,2k,p)$.  Applying Corollary~\ref{cor:k2kk} gives
\begin{equation}\label{eqn:D}
\4 D(2k,2k,p) \eq \4 \parbox{1in}{\epsfig{file=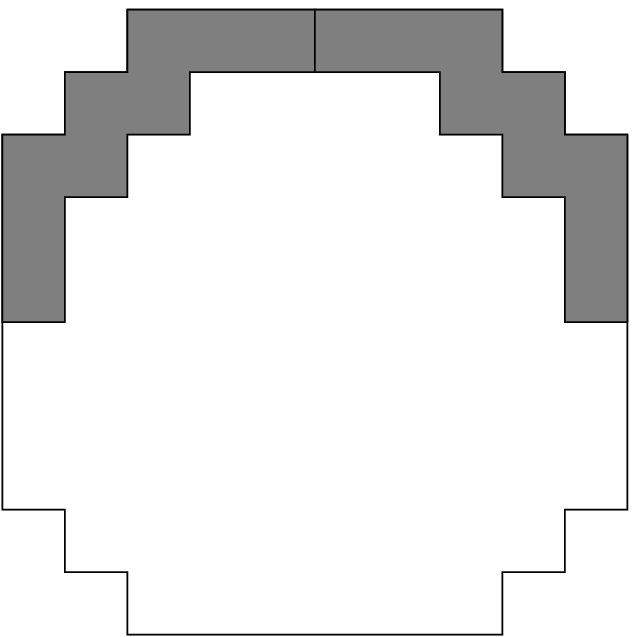,width=1in}} + \4 \parbox{1in}{\epsfig{file=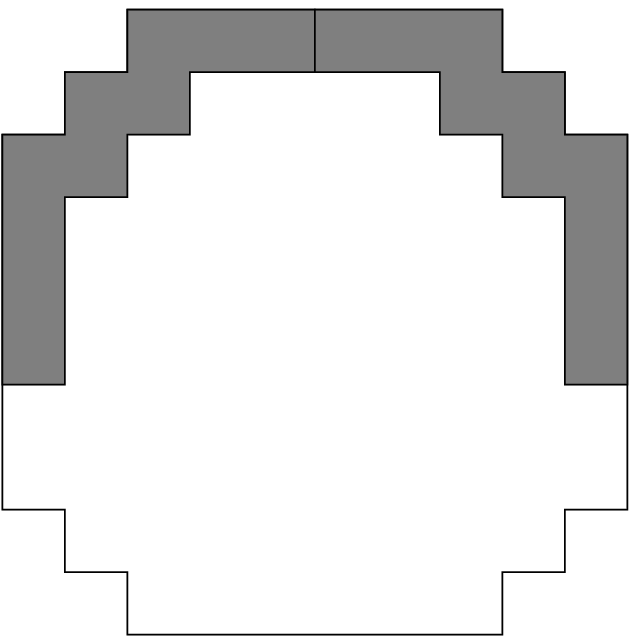,width=1in}}
\end{equation}
\noindent Now apply Corollary~\ref{cor:k2kk} to the first figure in equation~\eqref{eqn:D} and Theorem~\ref{thm:4ity theorem} to the second figure in \eqref{eqn:D}.  This gives
\begin{equation*}
\4 D(2k,2k,p) \eq 2 \4 D(2(k-1),2(k-1),p).
\end{equation*}

Therefore it is only necessary to determine $\4 D(2,2,p)$, where the region is the Aztec diamond of order $p$.  If $p \ge 2$, apply Corollary~\ref{cor:double wall} to $D(2,2,p)$, where the axis of symmetry is a diagonal of the region.  This shows that $\4 D(2,2,p) \eq 0$.  The remaining case is easy to enumerate: $D(2,2,1) = 2$. 

Note that this last case must be treated separately because any axis of symmetry in the $2 \times 2$ square intersects the regions needed to determine completeness of the corners.

\item Using Theorems~\ref{thm:parity theorem} and~\ref{thm:4ity theorem}, together with Lemma~\ref{lem:mod4mod2} and Corollary~\ref{cor:k2kk}, it is evident that
\begin{equation*}
\4 D(k,k+1,p) \eq 2\4 D(k-2,k-1,p) + \4 D(k-4,k-3,p)
\end{equation*}
\noindent for all $p$ and $k > 4$.  The cases when $k \in [1,4]$ are straightforward to check, and give the following results.
\begin{equation*}\begin{split}
\4 D(1,2,p) &= 1\\
\4 D(2,3,p) &\eq \left\{ \begin{array} {c@{\quad:\quad}l} 1 & p \text{ is even};\\ 3 & p \text{ is odd}. \end{array} \right.\\
\4 D(3,4,p) &\eq 2 + \4 D(2,3,p-1) \ \eq \ \4 D(2,3,p)\\
\4 D(4,5,p) &\eq 1 + 2 \4 D(2,3,p) \ \eq \ 3
\end{split}\end{equation*}

Therefore the value of $\# D(k,k+1,p)$ modulo 4 is described in the following table, depending on the value of $k$ and the parity of $p$.
\begin{equation*}
\begin{array}{r|cccccccccc}
k & 1 & 2 & 3 & 4 & 5 & 6 & 7 & 8 & 9 & 10\\
\hline
p \text{ odd} & 1 & 3 & 3 & 3 & 3 & 1 & 1 & 1 & 1 & 3\\
p \text{ even} & 1 & 1 & 1 & 3 & 3 & 3 & 3 & 1 & 1 & 1
\end{array}
\end{equation*}
\item If $k$ is odd, the region $D(k,k+2,p)$ has odd area, hence $\# D(k,k+2,p) = 0$.  Thus it remains to consider $D(2k,2k+2,p)$.  Once again, Theorems~\ref{thm:parity theorem} and~\ref{thm:4ity theorem} can be used together to show that
\begin{equation*}
\4 D(2k,2k+2,p) \eq 2\4 D(2k-2,2k,p) + \4 D(2k-4,2k-2,p).
\end{equation*}
The base cases here can be computed as follows, for $p \ge 2$.
\begin{equation*}\begin{split}
\4 D(2,4,p) &\eq \4 D(2,2,p-1)\\
\4 D(4,6,p) &\eq \4 D(2,2,p) + \4 D(2,2,p-1) + \4 D(4,4,p-1)
\end{split}\end{equation*}
Previous results for $\4 D(k,k,p)$ complete the proof.
\end{enumerate}
\end{proof}

Observe that Corollary~\ref{cor:D}(a) shows that the number of domino tilings of the Aztec diamond of order $p$ is a multiple of $4$ for all $p > 1$.  As mentioned above, the restriction on $p$ comes from the fact that if $p=1$, the regions checked to determine completeness of the reflective pair of corners would not be disjoint.  In fact, Elkies, Kuperberg, Larsen, and Propp show in \cite{elkies} that the exact number of domino tilings of this region is $2^{p(p+1)/2}$.

\end{document}